\documentclass[leqno, a4paper]{amsart}
\usepackage{amssymb}
\usepackage{hyperref}
\usepackage{color}
\newtheorem{theorem}{Theorem}[section]
\newtheorem{corollary}[theorem]{Corollary}

\newtheorem{proposition}[theorem]{Proposition}
\theoremstyle{definition}
\newtheorem{Definition}[theorem]{Definition}
\newtheorem{Remark}[theorem]{Remark}

\newtheorem{example}[theorem]{Example}
\numberwithin{equation}{section}

\begin{document}

%-------------------------------------------------------------------------
% editorial commands: to be inserted by the editorial office
%
%\firstpage{1} \volume{228} \Copyrightyear{2004} \DOI{003-0001}
%\seriesextraline{This is the Concrete Title of this Book\br H.E. R and S.T.C. W, Eds.}
%\firstpage{1}

\title[Generalized Drazin-Riesz invertible operators and $C_{0}$-semigroups]
{On the closed generalized Drazin-Riesz invertible operators and $C_{0}$-semigroups}

\author{Othman Abad and Hassane Zguitti}

\address{Department of Mathematics, Dhar El Mahraz Faculty of Science, Sidi Mohamed Ben Abdellah University, BO 1796 Fes-Atlas, 30003 Fez Morocco.}

\email{hassane.zguitti@usmba.ac.ma}
\email{othman.abad@usmba.ac.ma}


%\date{January 1, 2004}

\maketitle

\begin{abstract}
This paper is a continuation of our paper [Med. J. Math 19, Article number: 31 (2022)] in which we extended the notion of  generalized Drazin-Riesz invertible operators to closed operators. We  establish here, results relating the notion of closed generalized Drazin-Riesz invertibility with the theory of $C_{0}$-semigroups. Firstly, we generalize results obtained in the bounded case \cite{AbZg1} to the context of closed operators. Secondly, we  investigate when an infinitesimal generator $A$ of a given $C_{0}$-semigroup is  closed generalized Drazin-Riesz invertible. An application to $C_{0}$-groups and abstract second order differential equations is proposed, and an example of a  $C_{0}$-group with closed generalized Drazin-Riesz invertible  infinitesimal generator is given. 
\end{abstract}

%%% ----------------------------------------------------------------------

%%% ----------------------------------------------------------------------
%\tableofcontents

%------------------------------
\section{Introduction}
%--------------------------------
Right through this paper, $X$ is a complex Banach space of infinite dimension, $\mathcal{C}(X)$ and $\mathcal{L}(X)$  are the classes of all closed linear densely defined operators and bounded  linear operators on $X$, respectively. Let $A \in \mathcal{C}(X)$, the null space, the domain, and the range of $A$ are  respectively denoted by $\mathcal{N}(A)$,  $\mathcal{D}(A)$, and $\mathcal{R}(A)$. For $A\in\mathcal{C}(X)$ and $B\in\mathcal{L}(X)$,  $A$ and $B$ {\it are commuting} in $\mathcal{D}(A)$ if $\mathcal{R}(B)\subset\mathcal{D}(A)$ and $BAx=ABx$ for all $x \in \mathcal{D}(A)$; and then we write $AB=BA$. Let $\mathcal{D}^{\infty}(A):= \underset{n \in \mathbb{N}}{\bigcap} \mathcal{D}(A^{n}),$ where $$\mathcal{D}(A^{n})=\{x \in X \ : \ x,Ax,...,A^{n-1}x \in \mathcal{D}(A) \}, \mbox{ for all }  n \in \mathbb{N}.$$

For $A\in \mathcal{C}(X)$, let $\rho(A)=\{ \lambda\in\mathbb{C} \ : \ \lambda -A \mbox{ is bijective }\}$  and $\sigma(A)=\mathbb{C}\setminus \rho(A)$ be respectively, the {\it resolvent set} and the {\it spectrum} of $A$. By the graph theorem, if $\lambda\in\rho(A)$, then $(\lambda -A)^{-1}\in\mathcal{L}(X)$.
 Let $\sigma_{\infty}(A):=\sigma(A)\cup \{\infty\}$ be the compactified spectrum in $\mathbb{C}_{\infty}= \mathbb{C}\cup \{\infty\}$.
\smallskip

Here and elsewhere, for a subset $K$ of $\mathbb{C}$, we mean respectively by $iso\, K$, and $acc\, K$, the set of isolated points of $K$ and the set of accumulation points of $K$. $D(\lambda,r)$ denotes the open disc centered at $\lambda\in\mathbb{C}$ and with radius $r>0$ and its  corresponding closed disc is denoted by $\overline{D}(\lambda,r)$.
\smallskip

 An operator $A \in \mathcal{C}(X)$ is said to be of {\it finite ascent} if    $\mathcal{N}(A^n)=\mathcal{N}(A^{n+1})$ for some $n\in\mathbb{N}_{0}=\mathbb{N} \cup \{0\}$. The ascent of $A$ is denoted by $asc(A)$, and is defined to be the smallest $n\in\mathbb{N}_0$ such that $\mathcal{N}(A^n)=\mathcal{N}(A^{n+1})$. If no such  $n$ exists we set $asc(A)=\infty$. Similarly, $A$ is said to be of {\it finite descent} if there exists $n \in \mathbb{N}_{0}$ such that $\mathcal{R}(A^n)=\mathcal{R}(A^{n+1})$. The {\it descent} of $A$,  $dsc(A)$,  is the smallest $n\in\mathbb{N}_0$ provided that  $\mathcal{R}(A^n)=\mathcal{R}(A^{n+1})$. We set $dsc(A)=\infty$ if no such $n$ exists  \cite{Fakhfakh}.
 \smallskip

 $A \in \mathcal{C}(X)$ is {\it upper semi-Fredholm} (resp. {\it lower semi-Fredholm}) provided that $\mathcal{R}(A)$ is closed and $\dim\,\mathcal{N}(A)<\infty$ (resp. $\mbox{codim}\,\mathcal{R}(A)<\infty$). $A$ is {\it Fredholm} if $A$ is both  upper  and lower semi-Fredholm \cite{Fakhfakh}.  $A$ is {\it Browder} if it is Fredholm of finite ascent and descent \cite{Fakhfakh}. The {\it essential spectrum} $\sigma_e(A)$ and the {\it Browder spectrum} $\sigma_b(A)$ are respectively defined by $\sigma_e(A)=\{\lambda\in\mathbb{C}\,:\,A-\lambda\mbox{ is not Fredholm}\}$ and $\sigma_b(A)=\{\lambda\in\mathbb{C}\,:\,A-\lambda\mbox{ is not Browder}\}$. We respectively define the essential and Browder resolvents by $\rho_{e}(A)=\mathbb{C} \setminus \sigma_{e}(A)$ and $\rho_{b}(A)=\mathbb{C} \setminus \sigma_{b}(A)$. The set of all {\it Riesz points} of $A \in \mathcal{C}(X)$ is defined by $p_{00}(A):= \sigma(A) \setminus \sigma_{b}(A)= \sigma(A) \cap \rho_{b}(A)$ \cite{Fakhfakh}.
 \smallskip
 
\noindent If $\mathcal{R}(A)$ is closed and $\mathcal{N}(A)\subset\mathcal{R}(A^n)$ for all $n\in\mathbb{N}$, then $A \in \mathcal{C}(X)$ is called  a {\it Kato} operator \cite[Definition 2.1]{Fakhfakh}. 
\smallskip

An operator $T \in \mathcal{L}(X)$ is said to be quasinilpotent if $\sigma(T)=\{0\}$. $T \in  \mathcal{L}(X)$ is said to be Riesz if $\sigma_{e}(T) = \{0\}$ \cite{Aie}. 
\smallskip

\noindent We say that $A \in \mathcal{C}(X)$ has a reduction $(M,N)$, and we write $(M,N) \in Red(A)$, if $M$ and $N$ are closed subspaces of $X$ provided that $A(M\cap \mathcal{D}(A) ) \subset M$, $N \subset \mathcal{D}(A)$, $A(N) \subset N$, and $X=M \oplus N$. In this case, we write $A=A_{M} \oplus A_{N}$ 
with $\mathcal{D}(A_{M})=M \cap \mathcal{D}(A)$, $A_{M}=A$ in $\mathcal{D}(A_{M})$, and $A_{N}= A$ in $N$ \cite{AbZgMed} (see \cite[Definition V.5.1]{Taylor2} for the general case).
\smallskip

\noindent  $A\in \mathcal{C}(X)$ is said to admit a {\it generalized Kato-Riesz decomposition} (GKRD), if there is a pair $(M,N) \in Red(A)$ provided that $A_{M}$ is Kato and $A_{N}$ is bounded Riesz \cite{AbZgMed} (for bounded operators, see \cite{Ziv}), in this case we write $A \in GKRD(M,N)$.

\smallskip

Following \cite{NashZhao}, $A\in\mathcal{C}(X)$ is  said to be {\it Drazin invertible} if there is $B\in\mathcal{L}(X)$ provided that  $\mathcal{R}(B)$ and  $\mathcal{R}(I-AB)$ are contained in $\mathcal{D}(A)$ and
\begin{equation}\label{Bd} \ BAB=B,\ AB=BA, \mbox{ and }A(I-AB) \mbox{ is nilpotent}.
\end{equation}
Such a $B$ is unique if it exists. Moreover, the Drazin invertibility of $A$ is  equivalent to say that $0$ is a pole of finite order for the resolvent of $A$ \cite{NashZhao}. This notion was first introduced in 1958 by Drazin \cite{Dra} in semigroups and associative rings. The concept of Drazin invertibility was widely studied in matrix theory \cite{Ben,Camp} and in operator theory \cite{Car,King,Lay}.
\smallskip

In 2002,  Koliha and Tran \cite{Kol} introduced one of the most central generalizations of the Drazin inverse. $A\in\mathcal{C}(X)$ is {\it generalized Drazin invertible} if there is a  $B\in\mathcal{L}(X)$ provided that $\mathcal{R}(B)$ and  $\mathcal{R}(I-AB)$ are contained in $\mathcal{D}(A)$, and
\begin{equation}
\label{Cgd} \ BAB=B ,\ AB=BA \mbox{ and } \sigma(A(I-AB))=\{0\}.
\end{equation} 
If such a $B$ exists then it is unique and it is called the {\it generalized Drazin inverse} of $A$. Furthermore, $A \in \mathcal{C}(X)$ is generalized Drazin invertible if and only if $0 \notin acc\, \sigma(A)$ \cite{Kol}.
\smallskip

In \cite{AbZgMed} we extended to closed operators, the notion of bounded generalized Drazin-Riesz invertible operators introduced by  S. C. Živkovi\'c-Zlatanovi\'c and M.D. Cvetkovi\'c \cite{Ziv}:  $A\in \mathcal{C}(X)$ is {\it generalized Drazin-Riesz  invertible} provided that there is a $B \in \mathcal{L}(X)$ satisfying $\mathcal{R}(B)$ and  $\mathcal{R}(I-AB)$ are contained in $\mathcal{D}(A)$, and
\begin{equation}\label{DR} \ BAB=B,\ AB=BA \mbox{ and } A(I-AB) \mbox{ is bounded Riesz}.\end{equation}
We say that $B$ is a {\it generalized Drazin-Riesz inverse} of $A$.
\smallskip

The aim of this paper is to establish results relating the notion of closed generalized Drazin-Riesz invertibility with the theory of $C_{0}$-semigroups. To reach this aim, in Section \ref{section2}, we generalize results obtained in the bounded case \cite{AbZg1} to the context of closed operators. Especially, we show that $A \in \mathcal{C}(X)$ under other conditions is generalized Drazin-Riesz invertible if and only if $0 \notin acc\, \sigma_{b}(A)$. In Section \ref{section3}, a generalization of \cite[Theorem 7.1]{Gonzalez} is given by investigating when an infinitesimal generator $A$ of a given $C_{0}$-semigroup is  closed generalized Drazin-Riesz invertible. Also we give an integral formula to express a generalized Drazin-Riesz inverse of $A$. Finally, an application to $C_{0}$-groups and abstract second order differential equations is investigated and an example of a  $C_{0}$-group with unbounded infinitesimal generator which is closed generalized Drazin-Riesz invertible is given. 
%---------------------------------------------------------------------
\section{Closed generalized Drazin-Riesz inverse related to a spectral set containing 0}
\label{section2}
%----------------------------------------------------------------------
In the sequel, we denote by  $\widehat{{\mathcal{C}}}(X)$ the set of all $A \in \mathcal{C}(X)$  provided that $\rho_{e}(A) \neq \emptyset$ and for all sufficiently large $j \in \mathbb{N}$, $\mathcal{D}(A^{j+1})=\mathcal{D}(A^{j})$.
\smallskip

 The  next results will be widely applied in this paper.
\begin{theorem} [Theorem 3.6 \cite{AbZgMed}] \label{theo 2.3}
Let $A \in \widehat{\mathcal{C}}(X)$. Then the following conditions are equivalent:
\begin{enumerate}
\item[(i)] $A$ is closed generalized Drazin-Riesz invertible;
\item[(ii)] There exists $(M,N) \in Red(A)$ such that $A_{M}$ is closed invertible and $A_{N}$ is bounded Riesz;
\item[(iii)] There is $P \in \mathcal{L}(X)$ a projection which commute with $A$ such that $AP$ is bounded Riesz and $A+P$ is closed Browder;
\item[(iv)] There is $(M,N) \in Red(A)$ provided that  $A_{M}$ is closed Browder and $A_{N}$ is bounded Riesz;
\item[(v)] $0 \notin acc\, \sigma_{b}(A) $ and  $A \in GKRD(M,N)$;
 \end{enumerate}
\end{theorem}

If we assume in the previous theorem, that $A_N$ has infinite spectrum we get more
\begin{proposition} [Proposition 3.15 \cite{AbZgMed}] \label{Prop2.7}
Let $A\in \widehat{\mathcal{C}}(X)$, the next assertions are equivalent:
\begin{enumerate}
\item[(i)] $A \in GKRD(M,N)$, and there is a sequence with terms in $p_{00}(A)\setminus \{0\}$ converging to 0.
\item[(ii)] $A=A_{M} \oplus A_{N}$ such that $A_{N}$ is bounded Riesz with infinite spectrum and $A_{M}$ is closed invertible;
\end{enumerate}
\end{proposition}

\begin{corollary} [Corollary 3.16 \cite{AbZgMed}] \label{Corollary 2.4}
Let $A \in \widehat{\mathcal{C}}(X)$ with $0 \in acc \ \sigma(A)$. If $A$ is closed generalized Drazin-Riesz invertible, then, there is a sequence with terms in $p_{00}(A) \setminus \{0\}$  which converges to $0$.
\end{corollary}

Let $A\in\mathcal{C}(X)$ such that $\rho(A)\neq\emptyset$. A subset $\sigma$ of $\sigma_{\infty}(A)$ is said to be a {\it spectral set} if it is both open and closed (clopen) in the relative topology of $\sigma_{\infty}(A)$ as a subset of $\mathbb{C}_{\infty}$. If $\sigma$ is a bounded spectral set of $A$, then there is $(M,N)\in Red(A)$ such that $\sigma(A_M)=\sigma(A)\setminus \sigma$ and $\sigma(A_N)=\sigma$; and the projection $P$ with $\mathcal{N}(P)=M$ and $\mathcal{R}(P)=N$ is the {\it spectral projection} $P_\sigma$ of $A$ relative to $\sigma$. Moreover, $$P_\sigma={1\over 2\pi i}\int_\Gamma (\lambda-A)^{-1}d\lambda$$ where $\Gamma$ is a contour surrounding $\sigma$ and $\sigma(A)\setminus\sigma$ is in its outside and does not intersect it \cite[p. 207]{A.Taylor}.
\smallskip

\noindent In \cite[Definition 2.2]{Tran}, Tran introduced the notion of the Drazin inverse in relation with a spectral set containing zero as follow
\begin{Definition}[Definition 2.2\cite{Tran}] \label{Def2.2Tran}
Let $A\in \mathcal{C}(X)$ be a non-invertible operator with a bounded spectral set $\sigma$ containing $0$ and let $P_{\sigma}$ be the corresponding spectral projection. The {\it Drazin inverse of $A$ related to} $\sigma$ is defined by
$$A^{D,\sigma}=(A- \xi P_{\sigma})^{-1}(I-P_{\sigma}),$$ for some $\xi \in \mathbb{C} \ \mbox{such that } | \xi | > 2 r$ where $r=\underset{\lambda \in \sigma}{\sup} | \lambda |.$
\end{Definition}

\noindent In particular, if $\sigma=\{0\}$, then $A$ is generalized Drazin invertible and $A^{D,\sigma}$ is its generalized Drazin inverse (see \cite{Kol}).
\smallskip

Let $A \in \widehat{\mathcal{C}}(X)$ be closed generalized Drazin-Riesz invertible such that $0 \in \ acc \ \sigma(A)$. By Corollary \ref{Corollary 2.4} and Proposition \ref{Prop2.7}, there is $(M,N) \in Red(A)$ provided that $A_{M}$ is closed invertible and $A_{N}$ is bounded Riesz with infinite spectrum. Since $0 \in \rho(A_{M})$, there is $\delta > 0 $ such that $D(0,\delta) \subset \rho(A_{M})$. Hence, $$(\sigma(A_{N}) \cap \overline{p_{00}(A)} \cap D(0,\delta) ) \subset \rho(A_{M})$$ and $$\sigma(A_{N}) \setminus ( \overline{p_{00}(A)} \cap D(0, \delta)) \subset \sigma_{\infty}(A_{M}).$$
By the proof of Corollary \ref{Corollary 2.4}, $$(\sigma(A_{N}) \cap \overline{p_{00}(A)} \cap D(0,\delta) )= \{ 0, \lambda_{1}, \lambda_{2},...\}$$ where $(\lambda_{n})$ with terms in $p_{00}(A) \setminus\{0\}$ that converges to $0$. Without loss of generality, we assume $(|\lambda_{n}|)$ is a decreasing sequence.

Now for every positive integer $n$, we define the closed sets $\sigma_{n}$ and $\sigma'_{n}$ by

\begin{center} $\sigma_{n}=\{0,\lambda_{n+1},\lambda_{n+2},...\}$ and $\sigma'_{n}=\sigma_{\infty}(A)\setminus \sigma_{n}$.
\end{center}
Hence $\sigma_n$ is a bounded spectral set of $A$ containing zero.
\begin{theorem} \label{DRform}
Let $A \in \widehat{\mathcal{C}}(X)$ be closed generalized Drazin-Riesz invertible with $0 \in  acc \, \sigma(A)$. Let $n$ be sufficiently large such that $r_{n}= \underset{\lambda \in \sigma_{n} }{\sup} \mid \lambda \mid < \frac{1}{2}$. Then $$A^{D,\sigma_{n}}=(A-P_{\sigma_{n}})^{-1}(I-P_{\sigma_{n}}),$$
is a generalized Drazin-Riesz inverse of $A$.
\end{theorem}
\begin{proof}
Showing that $A^{D,\sigma_{n}}$ is a generalized Drazin-Riesz inverse of $A$, means that we have to verify the following conditions
\begin{enumerate}
\item[(i)] $A^{D,\sigma_{n}} \in \mathcal{L}(X)$, $\mathcal{R}(A^{D,\sigma_{n}}) \subset  \mathcal{D}(A)$ and $\mathcal{R}(I-AA^{D,\sigma_{n}}) \subset \mathcal{D}(A)$;
\item[(ii)] $AA^{D,\sigma_{n}}=A^{D,\sigma_{n}}A$, and $A^{D,\sigma_{n}}AA^{D,\sigma_{n}}=A^{D,\sigma_{n}}$;
\item[(iii)] $A(I-AA^{D,\sigma_{n}})$ is bounded Riesz.
\end{enumerate}

(i).  By virtue of \cite[Theorem 2.3]{Tran}, we have $A^{D,\sigma_{n}}=A_{\mathcal{N}(P_{\sigma_{n}})}^{-1} \oplus 0_{\mathcal{R}(P_{\sigma_{n}})}$, and then $A^{D,\sigma_{n}} \in \mathcal{L}(X)$. We have $P_{\sigma_{n}}=I-AA^{D,\sigma_{n}}$. Indeed, as  $P_{\sigma_{n}}$ commutes with $A$, we get that $P_{\sigma_{n}}(A-P_{\sigma_{n}})=(A-P_{\sigma_{n}})P_{\sigma_{n}}$.  Hence \\
 $P_{\sigma_{n}}(A-P_{\sigma_{n}})^{-1}=(A-P_{\sigma_{n}})^{-1}P_{\sigma_{n}}$. Also, 
\begin{align*}                 
                 AA^{D,\sigma_{n}}&= (A-P_{\sigma_{n}}+P_{\sigma_{n}})(A-P_{\sigma_{n}})^{-1}(I-P_{\sigma_{n}})\\
                 &=(A-P_{\sigma_{n}})(A-P_{\sigma_{n}})^{-1}(I-P_{\sigma_{n}})(I-P_{\sigma_{n}}) \\
                 &\ \ + P_{\sigma_{n}}(I-P_{\sigma_{n}})(A-P_{\sigma_{n}})^{-1}\\
                 &=I-P_{\sigma_{n}}.
\end{align*}
Therefore, $P_{\sigma_{n}}=I-AA^{D,\sigma_{n}}$.
\smallskip

\noindent Since $\mathcal{D}(A_{\mathcal{N}(P_{\sigma_{n}})})=\mathcal{N}(P_{\sigma_{n}})\cap \mathcal{D}(A)$, we get
$$\mathcal{R}(A^{D,\sigma_{n}})=\mathcal{R}(A_{\mathcal{N}(P_{\sigma_{n}})}^{-1}) \oplus \mathcal{R}(0_{\mathcal{R}(P_{\sigma_{n}})})=\mathcal{D}(A_{\mathcal{N}(P_{\sigma_{n}})}) \oplus \{0\}= \mathcal{D}(A_{\mathcal{N}(P_{\sigma_{n}})}).$$
 Consequently, $\mathcal{R}(A^{D,\sigma_{n}})\subset \mathcal{D}(A)$. Now, by \cite[Theorem 2.1 (i)]{Tran}, we have 
  $$\mathcal{R}(I-AA^{D,\sigma_{n}})=\mathcal{R}(P_{\sigma_{n}}) \subset \mathcal{D}(A).$$
\noindent (ii). As $A$ and $P_{\sigma_{n}}$ commute, we get that $A(A-P_{\sigma_{n}})=(A-P_{\sigma_{n}})A$, hence $A(A-P_{\sigma_{n}})^{-1}=(A-P_{\sigma_{n}})^{-1}A$. Thus
\begin{align*}
AA^{D,\sigma_{n}}&=A (A-P_{\sigma_{n}})^{-1}(I-P_{\sigma_{n}})\\
                 &=(A-P_{\sigma_{n}})^{-1}(I-P_{\sigma_{n}})A\\
                 &=A^{D,\sigma_{n}}A. 
\end{align*}
Also
\begin{align*} A^{D,\sigma_{n}}AA^{D,\sigma_{n}}&=(A-P_{\sigma_{n}})^{-1}(I-P_{\sigma_{n}})A A^{D,\sigma_{n}}\\
                                                           &=(A-P_{\sigma_{n}})^{-1}(I-P_{\sigma_{n}})(A-P_{\sigma_{n}}+P_{\sigma_{n}})A^{D,\sigma_{n}}\\
                                                           &=(A-P_{\sigma_{n}})^{-1}(A-P_{\sigma_{n}})(I-P_{\sigma_{n}})A^{D,\sigma_{n}}\\
                                                           &=(I-P_{\sigma_{n}})A^{D,\sigma_{n}}=A^{D,\sigma_{n}}.
\end{align*}
(iii). From $\mathcal{R}(P_{\sigma_{n}}) \subset \mathcal{D}(A)$ we have $AP_{\sigma_{n}}\in \mathcal{L}(X)$. Now by  \cite[Theorem 2.1]{Tran}, $\sigma(AP_{\sigma_{n}})=\sigma_{n}$. Since $\sigma_{n}$ consists of 0 and a sequence with terms in $p_{00}(A)\setminus \{0\}$ converging to 0, then for all $\lambda \in \mathbb{C}\setminus \{0\}$, $\lambda I - AP_{\sigma_{n}}$ is bounded Browder. Therefore, by \cite[Theorem 3.111]{Aie}, and as $(I-AA^{D,\sigma_{n}})=P_{\sigma_{n}}$, hence, $A(I-AA^{D,\sigma_{n}})$ is bounded Riesz.
\end{proof}

\noindent Following \cite{Mbekhta}, we define and denote the quasinilpotent part of a closed operator by $$H_{0}(A) := \{ x \in \mathcal{D}^{\infty}(A) \ : \ \underset{n \rightarrow +\infty}{\lim} ||A^{n}x||^{\frac{1}{n}}= 0 \}.$$
Also, by \cite[Theorem 3.1]{Kol2}, if $\mu_{0} \in iso\, \sigma(A)$, then 
\begin{equation}
\label{QuasinilpartandRangeofProj} \mathcal{R}(P_{\mu_{0}})= H_{0}(\mu_{0} I - A)\neq \{0\}. 
\end{equation}
The following theorem shows that a closed generalized Drazin-Riesz invertible operator may have more than one inverse.
\begin{theorem} \label{TheoDRinvdiff}
Let $A \in \widehat{\mathcal{C}}(X)$ be closed generalized Drazin-Riesz invertible with $0 \in acc\, \sigma(A)$. Let $n_{0},n_{1} \in \mathbb{N}$ be sufficiently large provided that $n_{0} < n_{1}$ and $r_{n_{0}} < \frac{1}{2}$. Then
$$A^{D,\sigma_{n_{0}}} \neq A^{D,\sigma_{n_{1}}}.$$
\end{theorem}
\begin{proof}
Suppose that $n_{0},n_{1} \in \mathbb{N}$ such that $n_{0}<n_{1}$ and $r_{n_{0}}< \frac{1}{2}$. Then
$$\sigma_{n_{0}}=\{0,\lambda_{n_{0}+1},...\}, \ \mbox{and } \sigma_{n_{1}}=\{0,\lambda_{n_{1}+1},...\}.$$
Thus we have $P_{\sigma_{n_{0}}}=I_{\mathcal{R}(P_{\sigma_{n_{0}}\setminus \sigma_{n_{1}}})} \oplus I_{\mathcal{R}(P_{\sigma_{n_{1}}})},$ with $\sigma(AP_{\sigma_{n_{0}}\setminus \sigma_{n_{1}}})=\sigma_{n_{0}}\setminus \sigma_{n_{1}}$. Hence $AP_{\sigma_{n_{0}}\setminus \sigma_{n_{1}}}$ is invertible.
In addition, $A_{\mathcal{R}(P_{\sigma_{n_{0}}})}=A_{\mathcal{R}(P_{\sigma_{n_{0}}\setminus \sigma_{n_{1}}})} \oplus A_{\mathcal{R}(P_{\sigma_{n_{1}}})}.$\\
Therefore $A=(A_{\mathcal{N}(P_{\sigma_{n_{0}}})}\oplus A_{\mathcal{R}(P_{\sigma_{n_{0}}\setminus \sigma_{n_{1}}})}) \oplus A_{\mathcal{R}(P_{\sigma_{n_{0}}})}.$\\
As $\mathcal{N}(P_{\sigma_{n_{1}}})=\mathcal{N}(P_{\sigma_{n_{0}}})\oplus \mathcal{R}(P_{\sigma_{n_{0}}\setminus \sigma_{n_{1}}})$, we obtain from  \cite [Theorem 2.3 ]{Tran} that
\begin{align*}
A^{D,\sigma_{n_{1}}} & =A^{-1}_{\mathcal{N}(P_{\sigma_{n_{1}}})}\oplus 0_{\mathcal{R}(P_{\sigma_{n_{1}}})} \\
                     & = A^{-1}_{\mathcal{N}(P_{\sigma_{n_{0}}})} \oplus A^{-1}_{\mathcal{R}(P_{\sigma_{n_{0}}\setminus \sigma_{n_{1}}})} \oplus 0_{\mathcal{R}(P_{\sigma_{n_{1}}})}.
\end{align*}
 By  virtue of \cite[Theorem V.9.1]{Taylor2} and (\ref{QuasinilpartandRangeofProj}), we have $$\mathcal{R}(P_{\sigma_{n_{0}}\setminus \sigma_{n_{1}}})= \underset{k \in \{n_{0}+1,...,n_{1}+2 \}}{\bigoplus}H_{0}(\lambda_{k}I-A) \neq \{0\}.$$ 
 Consequently $$A^{D,\sigma_{n_{0}}} - A^{D,\sigma_{n_{1}}} = 0_{\mathcal{N}(P_{\sigma_{n_{0}}})} \oplus -A^{-1}_{\mathcal{R}(P_{\sigma_{n_{0}}\setminus \sigma_{n_{1}}})} \oplus 0_{\mathcal{R}(P_{\sigma_{n_{1}}})} \neq 0.$$
\end{proof}

\begin{theorem} \label{thm3.3}
Let $A \in \widehat{\mathcal{C}}(X)$ be closed generalized Drazin-Riesz invertible with $0\in acc\, \sigma(A)$.  Let $R \in \mathcal{L}(X)$ be a Riesz operator satisfying $\mathcal{R}(R) \subset \mathcal{D}(A)$ and $AR=RA$. Then $A+R$ is closed generalized Drazin-Riesz invertible.
\end{theorem}
\begin{proof}
As $A \in \widehat{\mathcal{C}}(X)$ is closed generalized Drazin-Riesz invertible and $0 \in  acc\, \sigma(A)$, there exists $n \in \mathbb{N}$ such that $A^{D,\sigma_{n}}$ is a generalized Drazin-Riesz inverse for $A$, and $P_{\sigma_{n}}=I-AA^{D,\sigma_{n}}$ is the spectral projection of $A$ among $\sigma_{n}$.  Since $AR=RA$, we have $(\lambda I - A ) R= R( \lambda I -A)$ for all $\lambda \in \Delta$, also, we have $\mathcal{R}((\lambda I - A)^{-1})=\mathcal{D}(A)$. Hence $R(\lambda I - A )^{-1}=(\lambda I-A)^{-1}R$ in $X$.

 \cite[Lemma 3.4]{AbZgMed} ensures that $\mathcal{R}(P_{\sigma_{n}}) \subset \mathcal{D}(A)$, hence $RP_{\sigma_{n}}(X) \subset \mathcal{D}(A) $ because $\mathcal{R}(R) \subset \mathcal{D}(A)$, and $\mathcal{R}(P_{\sigma_{n}}R)\subset \mathcal{D}(A)$. Thus 
\begin{align*}
P_{\sigma_{n}}R & = (\frac{1}{2 \pi i}\int_{\Delta}(\lambda I - A)^{-1} d\lambda)R=\frac{1}{2 \pi i} \int_{\Delta} (\lambda I - A)^{-1}R d\lambda\\
                & = \frac{1}{2 \pi i}\int_{\Delta} R (\lambda I -A)^{-1}d \lambda= R(\frac{1}{2\pi}\int_{\Delta}(\lambda I - A)^{-1}d\lambda)=RP_{\sigma_{n}}, \ \mbox{ in } X
\end{align*}
where $\Delta$ is the frontier of a bounded Cauchy domain $D$, such that $\sigma_{n} \subset D$, and $\sigma_{n}^{'} \cap \bar{D} =\emptyset$.

\noindent Set $M=\mathcal{N}(P_{\sigma_{n}})$ and $N=\mathcal{R}(P_{\sigma_{n}})$, then $R=R_{M} \oplus R_{N}$ and so $A+R=(A_{M\cap \mathcal{D}(A)}+R_{M\cap \mathcal{D}(A)}) \oplus (A_{N}+R_{N})$.
As $A_{N}$ and $R_{N}$ commute and are bounded Riesz, $A_{N}+R_{N}$ is  Riesz by \cite[Theorem 3.112]{Aie}. \\
In addition, $A_{M}$ is closed invertible and $R_{M}$ is Riesz and  they both commute in $M \cap \mathcal{D}(A)$. Hence $A_{M}^{-1}$ commutes with $R_{M}$ in $M$, and using \cite[Theorem 3.112]{Aie}, $A_{M}^{-1}R_{M}$ is bounded Riesz. Thus $I_{M}+A_{M}^{-1}R_{M}$ is bounded Browder by \cite[Theorem 3.111]{Aie}. Therefore, $A_{M\cap \mathcal{D}(A)}+R_{M \cap \mathcal{D}(A)}=A_{M}(I_{M}+A_{M}^{-1}R_{M})_{M \cap \mathcal{D}(A)}=(I_{M}+A_{M}^{-1}R_{M})A_{M}$ is closed Browder by virtue of Theorem 2.1 \cite{AbZgMed}.
Finally, from Theorem \ref{theo 2.3}, we conclude that $A+R$ is closed generalized Drazin-Riesz invertible.
\end{proof}

We define and denote the generalized Drazin-Riesz spectrum of $A \in \mathcal{C}(X)$ as follow 
$$\sigma_{DR}(A):=\{ \lambda \in \mathbb{C} \ : \ \lambda I - A \mbox{ is not closed generalized Drazin-Riesz invertible} \},$$ 

As a direct consequence of Theorem \ref{thm3.3} we describe the generalized Drazin-Riesz spectrum of $A \in \widehat{\mathcal{C}}(X)$ as follow.

\begin{corollary}
Let $A \in \widehat{\mathcal{C}}(X)$ such that $0 \in acc\, \sigma(A)$. We have
$$\sigma_{DR}(A+R) = \sigma_{DR}(A),$$
for all $R$ a bounded Riesz operator such that $\mathcal{R}(R) \subset \mathcal{D}(A)$  and $AR=RA$.
\end{corollary}

In the following theorem we show that the condition "GKRD" in statement v) of Theoem \ref{theo 2.3} can be omitted.
\begin{theorem} \label{theo3.4}
Let $A \in \widehat{\mathcal{C}}(X)$. Then the following assertions are equivalent:\\
i) The closed operator $A$ is  generalized Drazin-Riesz invertible;\\ ii) $0 \notin  acc\, \sigma_{b}(A)$.
\end{theorem}
\begin{proof}
According to Theorem \ref{theo 2.3} it suffices to show the reverse sense. Suppose that $0 \notin  acc\, \sigma_{b}(A)$. If $0 \notin acc\, \sigma(A)$, then $A$ is closed generalized  Drazin invertible by \cite{Kol}. Thus, it is closed generalized  Drazin-Riesz invertible.\\
So assume that $0 \in  acc\, \sigma(A)$. If $0 \notin acc\ p_{00}(A)$ then $0 \notin acc\, \sigma_{b}(A) \cup acc\, p_{00}(A)=acc\, (\sigma_{b}(A) \cup p_{00}(A))=acc\, \sigma(A)$, which is a contradiction. Consequently, $0\in acc\ p_{00}(A)$.
\smallskip

Now, suppose that $0$ is not the only limit point of $p_{00}(A)$ in $\sigma_{\infty}(A)$. Therefore, we have
$$ 0 < \eta =\underset{\lambda \in (acc \, p_{00}(A))\setminus \{0\}}{\inf}(|\lambda|) \leq +\infty .$$
If it is not the case, therefore $\eta = 0$, hence there is a sequence $(\lambda_{n})$ in $(acc\, p_{00}(A))\setminus \{0\}$ which converges to $0$, as $n \rightarrow \infty$, this implies that $(\lambda_{n}) \subset (acc\, p_{00}(A))\setminus \{0\} \subset acc\, \sigma(A) \subset \sigma_{b}(A)$, therefore $0$ is a limit point of $\sigma_{b}(A)$, which is a contradiction.
\smallskip
Consequently, by putting $\eta_{1}=\min(\eta,\frac{1}{2})$, the elements of $\overline{D(0,\frac{\eta_{1}}{2})} \cap p_{00}(A)$ constitute a sequence with terms in $p_{00}(A)$ that converges to $0$ (notice that the adherence here is related to the topology of $\mathbb{C}_{\infty}$).
Hence $$\omega(A)=\overline{D}(0,\frac{\eta_{1}}{2})\cap \overline{p_{00}(A)}$$ and $$\sigma_{\infty}(A)\setminus \omega(A) = ((\sigma_{b}(A)\cup \{\infty\})\setminus \{0\})\cup (\overline{p_{00}(A)}\setminus \omega(A))$$ are spectral sets of $\sigma_{\infty}(A)$. Thus, by virtue of the spectral decomposition, there is $(M,N)\in Red(A)$, with $A_{M}$ is closed invertible and $A_{N}$ is bounded Riesz. Consequently, Theorem \ref{theo 2.3} leads to conclude that $A$ is closed generalized Drazin-Riesz invertible.

\noindent In the case where $acc\, p_{00}(A)= \{0\}$, we can take $\omega(A)=\overline{D}(0,\frac{1}{4}) \cap \overline{p_{00}(A)}$ to find that $\omega(A)$ consists of a sequence with terms in $p_{00}(A) \setminus \{0\}$ (converging to $0$) and of  $0$. Thus, we conclude that $\omega(A)$ and $\sigma_{\infty}(A)\setminus \omega(A)$ are spectral sets of $\sigma_{\infty}(A)$. By the spectral decomposition there is $(M,N) \in Red(A)$ with $A_{M}$ is closed invertible and $A_{N}$ is bounded Riesz. Theorem \ref{theo 2.3} ensures that $A$ is closed generalized Drazin-Riesz invertible.
\end{proof}

As a direct consequence of Theorem \ref{theo3.4}, we give the following result.
\begin{corollary}
Let $A \in \widehat{\mathcal{C}}(X)$, then we have $\sigma_{DR}(A)=acc\, \sigma_{b}(A)$.
\end{corollary}

\begin{theorem}\label{theo3.5}
Let $A \in \widehat{\mathcal{C}}(X)$ with $0 \in  acc \, \sigma(A)$ and $A$ is closed generalized Drazin-Riesz invertible. Then, there is some $n_{0} \in \mathbb{N}$ satisfying that
 $$A^{D,\sigma_{n_{0}}}=h(A),$$ 
 where $h$ is a holomorphic function on some open neighborhood of $\sigma_{\infty}(A)$, such that $h(\lambda)=0$ in some open neighborhood of $\sigma_{n_{0}}$ and $h(\lambda)=\lambda^{-1}$ in some open neighborhood of $\sigma_{\infty}(A) \setminus \sigma_{n_{0}}$. In that case
 \begin{align*}\sigma(A^{D,\sigma_{n_{0}}})&=\{0\} \cup \{ \lambda^{-1} \ / \ \lambda \in \sigma(A) \setminus \sigma_{n_{0}}   \} \\
                                           &=\{0\} \cup \{ \lambda^{-1} \ / \ \lambda \in \sigma_{\infty}(A) \setminus (\sigma_{n_{0}} \cup \{\infty\})   \}  .
 \end{align*}
\end{theorem}
\begin{proof}
Since $A \in \widehat{\mathcal{C}}(X)$ is closed generalized  Drazin-Riesz invertible, we have through Theorem \ref{theo3.4} that  $0 \notin acc \, \sigma_{b}(A)$. Hence there is some $\epsilon > 0$ such that $D(0,\epsilon) \cap \sigma_{b}(A)= \emptyset$, we are able to choose $n_{0}$ and $\epsilon$ such that $\sigma_{n_{0}} \cap D(0,\epsilon)=\sigma_{n_{0}}$, and $D(0,\epsilon)\cap \sigma_{n_{0}}^{'}=\emptyset$.
Consequently, $D(0,\epsilon)$ is an open neighborhood of $\sigma_{n_{0}}$, while $(\mathbb{C}_{\infty}\setminus \overline{D}(0, \epsilon))$ is an open neighborhood of $\sigma_{n_{0}}^{'}$.
 Let $f=1$ in $D(0,\epsilon)$ and $f=0$ in $\mathbb{C}_{\infty}\setminus \overline{D}(0,\epsilon)$, then by \cite[p. 321]{A.Taylor},  $P_{\sigma_{n_{0}}}=f(A)$.\\
By virtue of Definition \ref{Def2.2Tran} $A^{D,\sigma_{n_{0}}}=(A+P_{\sigma_{n_{0}}})^{-1}(I-P_{\sigma_{n_{0}}})=h(A)$, with $h(\lambda)=(1-f(\lambda))(\lambda + f(\lambda))^{-1}$.\\
The form of $h$ leads to conclude that $h(\lambda)=0$ in an open neighborhood of $\sigma_{n_{0}}$ and $h(\lambda)=\lambda^{-1}$ in an open neighborhood of $\sigma_{\infty}(A) \setminus \sigma_{n_{0}}$ with $\underset{\lambda \rightarrow \infty}{\lim} \lambda^{-1}=0$. By the spectral mapping theorem \cite[Theorem 9.1]{A.Taylor}, we obtain
$$\sigma(A^{D,\sigma_{n_{0}}})=\sigma(h(A))=h(\sigma_{\infty}(A))=\{0\} \cup \{ \lambda^{-1} \ / \ \lambda \in \sigma_{\infty}(A) \setminus (\sigma_{n_{0}} \cup \{\infty \} )\}.$$
\end{proof}

\begin{corollary} \label{IntegralrepofDRinv}
Let $A \in \widehat{\mathcal{C}}(X)$ be closed generalized  Drazin-Riesz invertible. Then, there is some $p \in \mathbb{N}$ such that $A^{D,\sigma_{p}}=\frac{1}{2 \pi i} \int_{\Gamma} \lambda^{-1}R(\lambda;A)d\lambda$, where $\Gamma$ is the clockwise oriented boundary of an unbounded Cauchy domain $D$ containing $\sigma'_{p}$, $\sigma_{p} \cap \bar{D} = \emptyset$, and $\Gamma$ is oriented clockwise.
\end{corollary}
\begin{proof}
By Theorem \ref{theo3.5}, there is $h$ a holomorphic function such that $h(\lambda)=0$ in a neighborhood $\mathcal{D}_{1}$ of $\sigma_{p}$ and  $h(\lambda)=\lambda^{-1}$ in a neighborhood $D$ of $\sigma'_{p}$ such that $h(A)=A^{D,\sigma_{p}}$.  

We choose $\mathcal{D}_{1}$ to be the disk $D(0,\epsilon)$ of center $0$ and radius $\epsilon > 0$, such that for all $\lambda \in \overline{D}(0,\epsilon), \ h(\lambda)=0$, and $\sigma_{p} \subset D(0,\epsilon)$ with $\sigma_{p} \cap C(0,\epsilon)= \emptyset$.
And let  $D$ to be $\mathbb{C} \setminus \overline{D}(0,r)$ with $r > \epsilon$, such that $D \cap \sigma'_{p}=\sigma'_{p}$, and for all $\lambda \in C(0,r)$, $h(\lambda)=\lambda^{-1}$, with $C(0,r) \cap \sigma'_{p}=\emptyset$.

It is obvious to see that $D(0,\epsilon)$ is a bounded Cauchy domain having the boundary $C(0,\epsilon)$ oriented counter clockwise. Also, by \cite[Theorem 3.2]{A.Taylor} and as $D(0,r)$ is a Cauchy domain, we know that $D$ is also an unbounded Cauchy domain which has the boundary  $\Gamma = C(0,r)$ but oriented clockwise.

Hence, by virtue of  \cite[Definition 4.1]{A.Taylor}, with considering $\underset{\lambda \rightarrow \infty}{\lim}h(\lambda)=0$, $h(\lambda)=0$ for all $\lambda \in \overline{\mathcal{D}_{1}}$, and $h(\lambda)=\lambda^{-1}$ for all $ \lambda \in D$, we get
\begin{align*}
 A^{D,\sigma_{p}}&=h(A)=\frac{1}{2 \pi i} \int_{\Gamma \cup C(0,\epsilon)} h(\lambda)R(\lambda,A)d\lambda \\
                 &= \frac{1}{2 \pi i} \int_{C(0,\epsilon)} h(\lambda) R(\lambda,A) d\lambda + \frac{1}{2 \pi i} \int_{\Gamma} h(\lambda)R(\lambda,A)d\lambda \\
                 &= \frac{1}{2 \pi i} \int_{\Gamma} \lambda^{-1} R(\lambda,A)d\lambda.
 \end{align*}
\end{proof}

In what follows, in the context of $0 \in \ iso \ \sigma_{b}(A)$, we mean by a suitable $n_{0} \in \mathbb{N}$ a large enough $n_{0}$ such that  $\sigma_{n_{0}}$ will be included in $D(0,r_{0})$ with $r_{0} <\frac{1}{4}$, $(D(0,r_{0}) \setminus \{0\}) \cap \sigma_{n_{0}}'=\emptyset$, and $\sigma_{n_{0}}'=\sigma_{\infty}(A) \setminus \sigma_{n_{0}}$.

We give the Laurent expansion for the resolvent of closed generalized Drazin-Riesz invertible operators.

\begin{theorem} \label{Theo2.12G}
Let $ A \in \widehat{\mathcal{C}}(X)$, $0 \in iso \, \sigma_{b}(A) $, and for a suitable $n_{0}\in \mathbb{N}$, consider  $A^{D,\sigma_{n_{0}}}$ to be a generalized Drazin-Riesz inverse of $A$ with $\sigma_{n_{0}}$ is a bounded spectral set having Riesz points of $A$ and $0$. Then, for every $\lambda\in D(0,(r(A^{D,\sigma_{n_{0}}}))^{-1})\setminus \overline{D(0,r(AP_{\sigma_{n_{0}}}))}$,
$$(\lambda I- A)^{-1}= \sum_{p=1}^{+\infty}\lambda^{-p}A^{p-1}(I-AA^{D,\sigma_{n_{0}}})-\sum_{p=0}^{+\infty}\lambda^{p}(A^{D,\sigma_{n_{0}}})^{p+1}.$$
\end{theorem}
\begin{proof}
By taking $\xi=-1$ in Definition \ref{Def2.2Tran}, we have $A^{D,\sigma_{n_{0}}}=(A+P_{\sigma_{n_{0}}})^{-1}(I-P_{\sigma_{n_{0}}})$. Then for all $\lambda \in D(0,(r(A^{D,\sigma_{n_{0}}}))^{-1})$, $\lambda A^{D,\sigma_{n_{0}}} - I$ is invertible.

Also, for all $\lambda \in D(0,(r(A^{D,\sigma_{n_{0}}}))^{-1})\setminus \overline{D(0,r(AP_{\sigma_{n_{0}}}))}  $, $\lambda I - AP_{\sigma_{n_{0}}}$ is invertible, because $|\lambda| > r(AP_{\sigma_{n_{0}}})$.
 
 Now, for every $\lambda \in \mathbb{C}$ such that $(r(AP_{\sigma_{n_{0}}})=)|\lambda_{n_{0}+1}| < |\lambda| < |\lambda_{n_{0}}|(=(r(A^{D,\sigma_{n_{0}}}))^{-1})$, we have
 \begin{align*}
 (\lambda I- A) & =(\lambda I - AP_{\sigma_{n_{0}}})P_{\sigma_{n_{0}}} + (\lambda (I-P_{\sigma_{n_{0}}})-(A+P_{\sigma_{n_{0}}}))(I-P_{\sigma_{n_{0}}}) \\
                & = \lambda  (  I - \lambda^{-1} AP_{\sigma_{n_{0}}})P_{\sigma_{n_{0}}} + (\lambda A^{D,\sigma_{n_{0}}}-I)(A+P_{\sigma_{n_{0}}})(I-P_{\sigma_{n_{0}}})
 \end{align*}
Finally, for every $r(AP_{\sigma_{n_{0}}}) < |\lambda| < (r(A^{D,\sigma_{n_{0}}}))^{-1}$
\begin{align*}
(\lambda I - A)^{-1}&= \lambda^{-1}( I - \lambda^{-1} AP_{\sigma_{n_{0}}})^{-1}P_{\sigma_{n_{0}}}+(\lambda A^{D,\sigma_{n_{0}}} - I)^{-1}A^{D,\sigma_{n_{0}}} \\
                    &= \sum_{p=1}^{+\infty}\lambda^{-p}A^{p-1}P_{\sigma_{n_{0}}}- \sum_{p=0}^{+\infty}\lambda^{p}(A^{D,\sigma_{n_{0}}})^{p} A^{D,\sigma_{n_{0}}}\\
                    &= \sum_{p=1}^{+\infty}\lambda^{-p}A^{p-1}(I-AA^{D,\sigma_{n_{0}}})-\sum_{p=0}^{+\infty}\lambda^{p}(A^{D,\sigma_{n_{0}}})^{p+1}
\end{align*}
\end{proof}
%......................................................................
\section{ $C_{0}$-semigroups and generalized Drazin-Riesz invertibility}
\label{section3}
The next theorem gives necessary conditions on the infinitesimal generator of a given $C_{0}$-semigroup to be closed generalized Drazin-Riesz invertible.
\begin{theorem} \label{Theo7.1G}
Let $(T(t))_{t\geq 0}$ be a bounded $C_{0}$-semigroup, and $A \in \widehat{\mathcal{C}}(X)$, its infinitesimal generator such that $0 \in acc\, \sigma(A)$. Let $P \in \mathcal{L}(X)$ be a non-zero projection such that 
\begin{enumerate}
\item[(i)] $T(t)P=PT(t), \ \forall t \geq 0$;
\item[(ii)] $\mathcal{R}(P) \subset \mathcal{D}(A)$;
\item[(iii)] $\| T(t)(I-P) \| \longrightarrow 0$ as $t \longrightarrow +\infty$;
\item[(iv)] $\sigma(AP)=\{0, \mu_{1},\mu_{2},....\}$ with $(\mu_{i})_{i\in \mathbb{N}}$ is a sequence with terms in $p_{00}(A)\setminus \{0\}$ converging to 0, such that $(|\mu_{i}|)_{i \in \mathbb{N}}$ is a decreasing sequence, and $|\mu_{1}| < \frac{1}{2}$.
\end{enumerate}
Then $A$ is closed generalized Drazin-Riesz invertible. Also, there is a spectral projection $Q$ of $A$ such that $A^{D,\sigma(AQ)}$ exists, and there are strictly positive constants $M, \mu$ such that 
\begin{align}
\label{C01} \| T(t)(I-P)\|  & \leq M e^{-\mu t} \quad  \forall \ t \geq 0, \mbox{ and} \\ 
A^{D,\sigma(AQ)}(I-P) &= (A-P)^{-1}(I-P) = - \int_{0}^{+\infty}T(t)(I-P)dt. \label{C02}
\end{align}
\end{theorem}
\begin{proof}
We have $APx=PAx, \ \forall x \in \mathcal{D}(A)$. Consider 
\begin{equation}
S(t)=T(t)e^{-tP}, \ \forall t \geq 0. \label{C03}
\end{equation}
By Lemma 4.1 \cite{Kol} $S(t)$ is a $C_{0}$-semigroup with the infinitesimal generator $C=A-P$. Also, with a straightforward calculation, we obtain
\begin{equation}
 \label{expoequation} e^{-tP}=I-P + e^{-t}P.
 \end{equation}
Hence
\begin{equation}
\| S(t) \| \leq \| T(t)(I-P) \|+\|T(t)\| \|P\| e^{-t} \underset{t \longrightarrow + \infty}{\longrightarrow} 0. \label{C04}
\end{equation}
Now, for all $x \in X$, for all $t \in [0,+\infty)$, we have
\begin{align}
 ||S(t)Px||&=||T(t)e^{-t}Px|| \label{EqEXP}   \\
         &\leq ||T(t)|| \ ||e^{-t}P|| \ ||x||  \\
         &\leq e^{-t} \ ||T(t)|| \  ||P||  \ ||x||.  
\end{align}
As $(T(t))_{t\geq 0}$ is a bounded $C_{0}$-semigroup, therefore there is $K > 0$ such that 
$$||S(t)Px|| \leq K e^{-t} ||x||. $$
Consequently
\begin{equation}
\label{InteqProj} \int_{0}^{+\infty}||S(t)Px||dt \leq K ||x||, \quad \forall x \in X.
\end{equation}
Set $S(t)=S_{1}(t) \oplus S_{2}(t)$, where $S_{1}(t)=S(t)_{|\mathcal{R}(P)}$ and $S_{2}(t)=S(t)_{|\mathcal{R}(I-P)}$, for all $t \in [0,+\infty)$. Then we get from (\ref{InteqProj})
\begin{equation}
\label{IntEqDec1} \int_{0}^{+\infty}||S(t)_{1}x||dt \leq K ||x||, \quad \forall x \in \mathcal{R}(P).
\end{equation}
Therefore by \cite[Theorem 3.8]{Neerven}, we find that
$$\omega_{0}(S_{1}) \leq \frac{-1}{K}<0,$$
where $\omega_{0}(S_{1})=\inf\{ \omega \in \mathbb{R} \ : \ \exists M_{1} > 0 \ \mbox{such that} \ ||S_{1}(t)|| \leq M e^{\omega t }, \ \forall t \geq 0\}$. Thus, there exists $M_{1} > 0$ such that $||S_{1}(t)|| \leq M_{1} e^{\omega_{0}(S_{1}) t}, \ \forall t \geq 0$.
\smallskip

On the other hand, from (\ref{C03}) and (\ref{expoequation}) we have $S(t)(I-P)=T(t)(I-P)$. Therefore $S_{2}(t)=T(t)_{|\mathcal{R}(I-P)}$. Hence $\underset{t \rightarrow +\infty}{\lim}||S_{2}(t)||=0$ by condition $(iii)$. Hence, for a sufficiently large $t$, say that there exists $t_{0}>0$ such that for all $t \geq t_{0}$, we have
$$||S_{2}(t)|| < 1.$$
Thus, we have 
$$\frac{\log(||S_{2}(t_{0})||)}{t_{0}}< 0.$$
Hence by \cite[Proposition 1.2.2]{Neerven}, we get
$$\omega_{0}(S_{2})=\frac{\log(r(S_{2}(t_{0})))}{t_{0}} < 0, \mbox{ where } r(S_{2}(t_{0})) \mbox{ is the spectral radius of } S_{2}(t_{0}).$$
Therefore, there exists $M_{2} > 0$ such that
\begin{equation}
\label{Eqwhichyield3.2} ||S_{2}(t)|| \leq M_{2}e^{\omega_{0}(S_{2})t}, \ \forall t \geq 0
\end{equation}
Now, as we have $T(t)(I-P)=S(t)(I-P)$, and as (\ref{Eqwhichyield3.2}) holds, we immediately deduce that $ || T(t)(I-P)||  \leq M e^{-\mu t} \quad  \forall \ t \geq 0$ by considering $M=M_{2}$ and $\mu =-\omega_{0}(S_{2})$.
\smallskip

\noindent Hence (\ref{C01}) is satisfied.
\smallskip
 
Now, due to the decomposition $S(t)=S_{1}(t) \oplus S_{2}(t)$ for all $t \geq 0$, we get 
$$||S(t)|| \leq 2 \max(||S_{1}(t)||,||S_{2}(t)||).$$
This implies the existence of $M_{1},M_{2}$ strictly positive, such that
$$||S(t)|| \leq 2 \max(M_{1}e^{\omega_{0}(S_{1})t},M_{2}e^{\omega_{0}(S_{2})t}).$$
We obtain for all $t \geq 0$, by setting $\mu= \min(-\omega_{0}(S_{1}),-\omega_{0}(S_{2}))$ and \\
$M_{0}=2\max(M_{1},M_{2})$, the following inequality 
\begin{equation}
\label{integrability of S(t)} ||S(t)||\leq M_{0} e^{-\mu t}.
\end{equation}
As $\omega_{0}(S) \leq - \mu < 0$, and considering $s(C)=\sup\{ Re(\lambda) \ : \ \lambda \in \sigma(C) \},$
by virtue of  \cite[Proposition 1.2.1]{Neerven}, we have $s(C) \leq \omega_{0}(S)$.
\smallskip

Hence, the spectrum of the generator $C$ of $S(t)$ lies in the open left-half plan of $\mathbb{C}$. Thus, $C$ is invertible.

 As $C=A-P$ is invertible, this implies that $-A+P$ is invertible. Also $-AP$ is bounded Riesz. Hence by virtue of Theorem \ref{theo 2.3}, $-A$ is generalized Drazin-Riesz invertible, therefore $A$ is generalized Drazin-Riesz invertible.
 
There exists $\theta >0$ such that $(\lambda I -C)$ is closed invertible if $| \lambda | < \theta$, because $0 \in \rho(C)$.
\smallskip

Now, as $\omega_{0} < 0$, by \cite[Theorem 1.1.4 and page 6]{Neerven}, we conclude that
\begin{equation}
C^{-1}=-\int_{0}^{+\infty}S(t)dt. \label{C010}
\end{equation}
 As a consequence of (\ref{C01}), we obtain that $T(t)(I-P)$ is integrable over $[0,\infty)$. Thus, by virtue of (\ref{C010}), we obtain
\begin{align}
\int_{0}^{\infty} T(t)(I-P)dt&=\int_{0}^{\infty}S(t)(I-P)dt=(\int_{0}^{\infty}S(t)dt) (I-P) \\
                             &=-C^{-1}(I-P)=-(A-P)^{-1}(I-P). \label{C011}
\end{align}
Now we show the existence of the projection $Q$. If $P$ is a spectral projection, then $\sigma(AP)$ is a spectral set of $A$, and in this case we can take $P=Q$.
\smallskip

If $P$ is not a spectral projection, as $\sigma(AP)$ is composed of $(\mu_{i})_{i \in \mathbb{N}}$ and $0$, where $(\mu_{i})$ is a converging sequence to 0 with terms in $p_{00}(A)\setminus \{0\}$, then, there exists $n_{0} \in \mathbb{N}$ such that $\sigma=D(0,\theta) \cap \sigma(AP)=\{0,\mu_{n_{0}+1},....\}$ which is closed. Also, since $A(I-P)=C(I-P)$, then, $\sigma_{\infty}(A) \setminus \sigma = \{\mu_{1}, \mu_{2},..., \mu_{n_{0}}\} \cup \sigma_{\infty}(C(I-P))$ is closed. Therefore, $\sigma$ is a spectral set of $\sigma_{\infty}(A)$ containing 0. Thus we take $Q=P_{\sigma}$ which satisfies conditions (i)-(iv) of \cite[Theorem 2.1]{Tran}.
\smallskip

Now, let us show that $QP=PQ=Q$.

\noindent  As $P$ commutes with $A$ and $Q= \frac{1}{2 \pi i} \int_{C(0,\theta)} (\lambda I - A)^{-1}d\lambda$, we conclude that $QP=PQ$. On the other hand, taking into account that for all $\lambda \in D(0,\theta)\setminus \sigma$, we have
$$(\lambda I - A )^{-1}=(\lambda I - A )^{-1}P + (\lambda - C )^{-1}(I-P),$$
we conclude that 
\begin{align*} 
Q&= \frac{1}{2 \pi i} \int_{C(0,\theta)} ( \lambda I - A )^{-1}d\lambda \\
 &=  \frac{1}{2 \pi i} \int_{C(0,\theta)} ( \lambda I - A)^{-1}P d\lambda + \frac{1}{2 \pi i} \int_{C(0,\theta)}(\lambda I - C)^{-1}(I-P)d\lambda \\
 &= \frac{1}{2 \pi i} \int_{C(0,\theta)} ( \lambda I - A)^{-1} d\lambda P = QP
\end{align*}
As a direct consequence, we get $(P-Q)(I-P)=(I-P)(P-Q)=0$. Since 
$$
A-Q=A-P+(P-Q)=(A-P)Q+(A-P)(I-Q)+(P-Q),
$$
we find $(A-Q)(I-P) = (A-P)(I-P)(I-Q)$. Therefore,
$$(A-P)^{-1}(I-P)=(A-Q)^{-1}(I-Q)(I-P)=A^{D,\sigma(AQ)}(I-P).$$
Finally we conclude that
$$-\int_{0}^{\infty} T(t)(I-P)dt = A^{D,\sigma(AQ)}(I-P)=(A-P)^{-1}(I-P).$$
\end{proof}

\begin{Remark}\rm
	It is easy to see that in Theorem \ref{Theo7.1G}, $(A-P)^{-1}(I-P)$  is a generalized Drazin-Riesz inverse of $A$. Then, the second equality in (\ref{C02}) gives an integral formula of a generalized Drazin-Riesz inverse of $A$.
\end{Remark}

Let us consider the abstract differential equation of second order
\begin{equation}
\frac{d^{2}x(t)}{dt^{2}}=A^{2}x(t)+f(t), \quad t\in [0,\delta]. \label{Absteq}
\end{equation}
A function $x: [0,\delta] \longrightarrow X$ is a solution of (\ref{Absteq}), if it takes values in $\mathcal{D}(A^{2})$, is twice continuously differentiable and satisfies (\ref{Absteq}) on $[0,\delta]$. A slight observation leads to see that if $x$ is twice continuously differentiable on $[0,\delta]$, then $Bx$ is also twice continuously differentiable on $[0,\delta]$ for every $B \in \mathcal{L}(X)$ \cite{Tran}.\\
Let $f$ be a continuous function on $[0,\delta]$. We define a primitive of $f$ by
$$F(t)=\int_{0}^{t}f(s)ds, \ \mbox{for } t \in [0,\delta].$$
Remark that $\underset{t \in [0,\delta]}{\sup} \| F^{(n)}(t) \| \leq M \delta^{n}$ for each $n \in \mathbb{N}$ where $F^{(n)}$ is the nth primitive of $f$  and $M=\underset{t\in [0,\delta]}{\sup} \| f(t) \|$ \cite{Tran}.

We motivate our results by the use of \cite[Theorem 3.1]{Tran} made by T.D. Tran, in the special case where the infinitesimal generator of the given bounded strongly continuous group is a closed generalized Drazin-Riesz invertible operator having $0$ as a limit point of its spectrum.

\begin{theorem}
Let $(V(t))_{t \in \mathbb{R}}$ be a bounded strongly continuous group, and $A \in \widehat{\mathcal{C}}(X)$ be its infinitesimal generator which is closed generalized Drazin-Riesz invertible with $0 \in acc\, \sigma(A)$. Then, there exists a spectral set $\sigma$ and $r > 0$ such that  $0 \in \sigma \subset D(0,r)$.  If $f$ is continuously differentiable on $[0,\delta]$,  then the unique solution of Equation (\ref{Absteq}) relative to $\sigma$ with initial conditions $x(0)=u_{0}$ and $\frac{d}{dt}_{\mid t=0}x(t)=v_{0}$ can be expressed by
\begin{align}
\label{3.2} x(t) & = \sum_{j=1}^{\infty}A^{2(j-1)}P_{\sigma}F^{(2j)}(t) \\
                 & + \frac{1}{2}(V(t)-V(-t))(I-P_{\sigma})u_{0}+\frac{1}{2}A^{D,\sigma}(V(t)-V(-t))(I-P_{\sigma})v_{0} \\
                 & + \int_{0}^{t}A^{D,\sigma}(V(t-s)-V(s-t))(I-P_{\sigma})f(s)ds
\end{align}
for each $t \in [0,r^{-1}]$ where $D(0,r) \cap (\sigma(A) \setminus \sigma)= \emptyset$ and $\sigma \subsetneq D(0,r)$, provided that $u_{0} \in \mathcal{D}(A^{2})$ satisfies
$$\sum^{\infty}_{j=1}A^{2(j-1)}P_{\sigma}F^{(2j)}(0)=P_{\sigma}u_{0},$$
and $v_{0}\in \mathcal{D}(A)$ satisfies
$$\sum_{j=1}^{\infty}A^{2(j-1)}P_{\sigma}F^{(2j-1)}(0)=P_{\sigma}v_{0},$$
where $F^{(j)}$ is the jth primitive of $f$.
\end{theorem}
\begin{proof}
We consider by the construction in pages 03-04, the spectral set $\sigma = \sigma_{n}$ of $A$ for a sufficiently large $n\in \mathbb{N}$. If $f$ is continuously differentiable on $[0,\delta]$, we obtain the desired solution of Equation (\ref{Absteq}) relative to $\sigma$ by applying \cite[Theorem 3.1]{Tran}.
\end{proof}
\begin{Remark}\rm
For a bounded strongly continuous group $(V(t))_{t \in \mathbb{R}}$ such that $A$ is its infinitesimal generator having $0$ as a limit point of its spectrum. As $A$ is also the infinitesimal generator of $(V_{+}(t))=\{ V(t) \ : \ t \in [0,+\infty) \}$, if there exists a projection $P$ such that $A$, $(V_{+}(t))$ and $P$ satisfy the conditions of Theorem \ref{Theo7.1G}, then $A$ is closed generalized Drazin-Riesz invertible. Consequently, we can apply the last theorem on $A$ and $(V(t))$ if $f$ is continuously differentiable on $[0,\delta]$ to solve Equation (\ref{Absteq}).
\end{Remark}
In the following, we construct an example of a bounded $C_{0}$-group with an unbounded infinitesimal generator which is generalized Drazin-Riesz invertible.
\begin{example}\rm
On $L^{2}(\mathbb{R})$,let $B_{1}$ be the operator defined by $B_{1}f(x)=e^{\frac{x^{2}}{4}}f(x)$ for all $f \in \mathcal{D}(B_{1})$ where
$$\mathcal{D}(B_{1})=\{ f\in L^{2}(\mathbb{R}) \ : \ e^{\frac{x^{2}}{4}}f, \ e^{\frac{x^{2}}{2}} \tilde{f} \in L^{2}(\mathbb{R}) \},$$
with $\tilde{f}$ is the Fourier Transform of $f$. $B_{1}$ is essentially self-adjoint, hence, $\bar{B}_{1}$ is a closed self-adjoint densely defined opreator such that

\noindent $\mathcal{D}(\bar{B}_{1})=\mathcal{D}(B^{*}_{1})=\{ f\in L^{2}(\mathbb{R}) \ : \ e^{\frac{x^{2}}{4}}f \in L^{2}(\mathbb{R}) \}$  and
$\bar{B}_{1}f(x)=e^{\frac{x^{2}}{4}}f(x)$ for all $f \in \mathcal{D}(\bar{B}_{1})$. Following \cite[pages 407-408]{Mortad} we have $\mathcal{D}(B_{1}^{2})=\{0\}$, hence
\begin{align*}
\mathcal{D}(\bar{B}_{1}^{2})&=\{f \in L^{2}(\mathbb{R}) \ : \ \exists (f_{n}) \subset \mathcal{D}(B_{1}^{2}) \ \mbox{such that  } f_{n} \rightarrow f \mbox{ and } B_{1}^{2}f_{n} \mbox{converges} \}\\
                            &=\{0\}.
\end{align*}
Thus for all integer $p \geq 2$, we have $\mathcal{D}(\bar{B}^{p})=\mathcal{D}(\bar{B}^{p+1})$.  Also, $-1  \in \rho(\bar{B}_{1})$. Indeed, for all $f \in \mathcal{D}(\bar{B}_{1})$, $(- I - \bar{B}_{1}) f(x) = -(1 + e^{\frac{x^{2}}{4}})f(x)$.
\smallskip

\noindent Hence $(1+ e^{\frac{x^{2}}{4}})^{-1}[ ( I + \bar{B}_{1}) f(x) ]=( I + \bar{B}_{1})[(1+ e^{\frac{x^{2}}{4}})^{-1}f(x)] =f(x)$. Remark that if $f \in L^{2}(\mathbb{R})$, we have $\frac{f}{1+e^{\frac{x^{2}}{4}}} \in L^{2}(\mathbb{R})$, thus we consider for all $f \in L^{2}(\mathbb{R})$
$$-(I+\bar{B}_{1})^{-1}f(x)=-(1+ e^{\frac{x^{2}}{4}})^{-1}f(x), \ \mbox{for all } f \in \mathcal{D}(\bar{B}_{1}).$$
\noindent Thus $-(I+\bar{B}_{1}) \in \widehat{\mathcal{C}}(L^{2}(\mathbb{R}))$, and it is self-adjoint and invertible. 
\smallskip

\noindent Now, we define $B_{2}: \ell^{2}(\mathbb{N}) \rightarrow \ell^{2}(\mathbb{N})$, as follows
$$B_{2}(x_{1},x_{2},....)=(x_{1},\frac{x_{2}}{2},\frac{x_{3}}{3},....),$$
$B_{2}$ is bounded, compact and self-adjoint. Hence $B=-(I+\bar{B}_{1}) \oplus B_{2}$ is closed generalized Drazin-Riesz invertible and self-adjoint in $H$, also it belongs to $\widehat{\mathcal{C}}(H)$, as desired. 

 \noindent Now by Proposition 6.1 and Theorem 6.2 of \cite{Schm}, $U= \{ U(t):= e^{itB} \, : \, t\in \mathbb{R} \}$ is a strongly continuous one-parameter unitary group on $H$, and $iB$ is its infinitesimal generator with domain:
$$\mathcal{D}(iB)=\mathcal{D}(B)=\{ x \in H: \frac{d}{dt} |_{t=0} U(t)x=\underset{h \longrightarrow 0^{+}}{\lim} h^{-1}(U(h)-I)x \ \mbox{exists} \}.$$
\end{example}

%---------------------------------------------------------------------

\end{document}